\documentclass[12pt]{article}%
\usepackage[intlimits]{amsmath}
\usepackage{amssymb}
\usepackage{latexsym}
\usepackage[T1]{fontenc}
\usepackage[sc]{mathpazo}
\usepackage{color}
\usepackage[colorlinks]{hyperref}
\usepackage{amsfonts}
\usepackage{microtype}
\usepackage{graphicx}%
\definecolor {refcol}{RGB}{40,0,255}
\hypersetup{colorlinks=true,allcolors=refcol}
\makeatletter
\newfont{\footsc}{cmcsc10 at 8truept}
\newfont{\footbf}{cmbx10 at 8truept}
\newfont{\footrm}{cmr10 at 10truept}
\newtheorem{theorem}{Theorem}

\newtheorem{corollary}[theorem]{Corollary}

\newtheorem{problem}[theorem]{Problem}
\newtheorem{proposition}[theorem]{Proposition}

\newenvironment{proof}[1][Proof]{\noindent{\textbf {#1}  }}  {\hfill$\Box$\bigskip}
\begin{document}

\title{\textbf{On the minimum trace norm of }$\left(  0,1\right)  $\textbf{-matrices}}
\author{V. Nikiforov\thanks{Department of Mathematical Sciences, University of
Memphis, Memphis TN 38152, USA; \textit{email: vnikifrv@memphis.edu}} \ and N.
Agudelo\thanks{Instituto de Matem\'{a}ticas, Universidad de Antioquia,
Medell\'{\i}n, Colombia; \textit{email: nagudel83@gmail.com}}}
\maketitle

\begin{abstract}
The trace norm of a matrix is the sum of its singular values. This paper
presents results on the minimum trace norm $\psi_{n}\left(  m\right)  $ of
$\left(  0,1\right)  $-matrices of size $n\times n$ with exactly $m$ ones. It
is shown that:

(1) if $n\geq2$ and $n<m\leq2n,$ then $\psi_{n}\left(  m\right)  \leq
\sqrt{m+\sqrt{2\left(  m-1\right)  }}$, with equality if and only if $m$ is a prime;

(2) if $n\geq4$ and $2n<m\leq3n,$ then $\psi_{n}\left(  m\right)  \leq
\sqrt{m+2\sqrt{2\left\lfloor m/3\right\rfloor }}$, with equality if and only
if $m$ is a prime or a double of a prime;

(3) if $3n<m\leq4n,$ then $\psi_{n}\left(  m\right)  \leq\sqrt{m+2\sqrt{m-2}}%
$, with equality if and only if there is an integer $k\geq1$ such that
$m=12k\pm2$ and $4k\pm1,6k\pm1,12k\pm1$ are primes.

\medskip



\end{abstract}

\section{Introduction}

The \emph{trace norm} $\left\Vert A\right\Vert _{\ast}$ of a matrix $A$, that
is to say, the sum of the singular values of $A$,\emph{ }is one of the most
studied matrix parameters. In particular, the trace norm of the adjacency
matrices of graphs has been long investigated under the name \emph{graph
energy,} a concept introduced by Gutman in \cite{Gut78}; for an overview of
this vast research, see \cite{GLS12}.

A number of extremal problems about the trace norm of matrices have been
presented in the survey \cite{Nik16}, including many upper bounds on
$\left\Vert A\right\Vert _{\ast}$. Since lower bounds on $\left\Vert
A\right\Vert _{\ast}$ have not been studied in comparative detail, in this
paper we initiate the study of the minimum trace norm of square $\left(
0,1\right)  $-matrices with given number of ones. This topic turns out to be
both fascinating and hard; in particular, we show that some rather simple
questions are tantamount to unsolved problems about prime numbers.

Thus, let the integers $n$ and $m$ satisfy $n\geq2$ and $1\leq m\leq n^{2}$,
write $\mathbb{Z}_{n}\left(  m\right)  $ for the set of $\left(  0,1\right)
$-matrices of size $n\times n$ with exactly $m$ ones, and set%
\[
\psi_{n}\left(  m\right)  =\min\left\{  \left\Vert A\right\Vert _{\ast}%
:A\in\mathbb{Z}_{n}\left(  m\right)  \right\}  .
\]
We are interested in the following natural problem:

\begin{problem}
\label{pr1}Find $\psi_{n}\left(  m\right)  $ for all admissible $n$ and $m.$
\end{problem}

It is not hard to see that $\psi_{n}\left(  m\right)  \geq\sqrt{m}$; in fact,
writing $\left\vert A\right\vert _{2}$ for the Frobenius norm of a matrix $A$,
one can come up with the following simple result (see, e.g., Theorem 4.3 of
\cite{Nik16}):\medskip

\emph{If }$A$\emph{ is a complex matrix, then} $\left\Vert A\right\Vert
_{\ast}\geq\left\vert A\right\vert _{2}$.\emph{ Equality holds if and only if
the rank of }$A$\emph{ is }$1$\emph{.\medskip}

It is trivial to construct a complex matrix of rank $1$ with arbitrary
$\left\vert A\right\vert _{2}$, but this is not always possible should the
matrix belong to $\mathbb{Z}_{n}\left(  m\right)  ,$ e.g., $\mathbb{Z}%
_{3}\left(  5\right)  $ contains no matrix of rank $1$. Consequently, finding
$\psi_{n}\left(  m\right)  $ turns out to be a subtle and challenging problem,
sometimes leading to extremely difficult number-theoretical questions. We
solved Problem \ref{pr1} for $m\leq3n$ and partially solved it for
$3n<m\leq4n$; even at that simple stage it becomes clear that the full
solution of Problem 1 is beyond the reach of present day mathematics.

Before stating our main results, note that the case $1\leq m\leq n$ is
trivial, as any $n\times n$ matrix with $m$ ones in a single row or column
implies that $\psi_{n}\left(  m\right)  =\sqrt{m}$. In contrast, the cases
$n<m\leq2n$ and $2n<m\leq3n$ are far from obvious.

As mentioned above, $\psi_{n}\left(  m\right)  =\sqrt{m}$ if and only if
$\mathbb{Z}_{n}\left(  m\right)  $ contains a matrix of rank $1$, that is to
say, if there exist integers $a$ and $b$ such that $2\leq a\leq b\leq n$ and
$m=ab.$ To refer to these special cases, let $C(n)$ denote the set of all $k$
such that $k=ab$ for some integers $a$ and $b$ satisfying $2\leq a\leq b\leq
n$. The structure of the set $C(n)$ is nontrivial in general, but it can be
described explicitly if $m\leq3n$. For example, if $n<m\leq2n$, one finds that
$m\in C\left(  n\right)  $ if and only if $m$ is not a prime. Accordingly, we
come up with the following statement:

\begin{theorem}
\label{t1}Let $n\geq2$ and $n<m\leq2n$. If $m$ is a prime, then%
\[
\psi_{n}\left(  m\right)  =\sqrt{m+\sqrt{2\left(  m-1\right)  }}\text{;}%
\]
otherwise $\psi_{n}\left(  m\right)  =\sqrt{m}$.
\end{theorem}

The main strategy of the proof of Theorem \ref{t1} is taking a matrix
$A\in\mathbb{Z}_{n}\left(  m\right)  $ with $\left\Vert A\right\Vert _{\ast
}=\psi_{n}\left(  m\right)  $ and showing that $A$ cannot contain certain
small submatrices, forcing eventually that $A$ has rank $2$ and has a
particular shape.\ Similar ideas allow to tackle also the case $2n<m\leq3n,$
which implies that $m\in C\left(  n\right)  $ if and only if $m$ is not a
prime or a double of a prime. The latter fact led us to the following statement.

\begin{theorem}
\label{t2} Let $n\geq4$ and $2n<m\leq3n.$ If $m$ is a prime or a double of a
prime, then%
\[
\psi_{n}(m)=\sqrt{m+2\sqrt{2\lfloor{m/3}\rfloor}}\text{;}%
\]
otherwise $\psi_{n}\left(  m\right)  =\sqrt{m}$.
\end{theorem}

Let us also add that in the proofs of Theorems \ref{t1} and \ref{t2} we
effectively determine the matrices $A$ with $\left\Vert A\right\Vert _{\ast
}=\psi_{n}(m)$.

Next, we consider the case $3n<m\leq4n,$ which is by far more difficult
than\ the previous two. Our result is not as clear-cut as before, for reasons
explained below.

\begin{theorem}
\label{t3}If $3n<m\leq4n$, then%
\begin{equation}
\psi_{n}(m)\leq\sqrt{m+2\sqrt{m-2}}\text{.} \label{mb}%
\end{equation}
If $n\geq6,$ equality holds in (\ref{mb}) if and only if there exists a
positive integer $k$ such that one of the following conditions is met:

(a) $m=12k+2$ and $4k+1,$ $6k+1,$ $12k+1$ are primes;

(b) $m=12k-2$ and $4k-1,$ $6k-1,$ $12k-1$ are primes.
\end{theorem}

It is not difficult to find examples for which bound (\ref{mb}) can be
improved. For instance, taking $n=7,$ $m=26,$ and letting
\[
A=\left[
\begin{array}
[c]{ccccccc}%
1 & 1 & 1 & 1 & 1 & 1 & 0\\
1 & 1 & 1 & 1 & 1 & 0 & 0\\
1 & 1 & 1 & 1 & 1 & 0 & 0\\
1 & 1 & 1 & 1 & 1 & 0 & 0\\
1 & 1 & 1 & 1 & 1 & 0 & 0\\
0 & 0 & 0 & 0 & 0 & 0 & 0\\
0 & 0 & 0 & 0 & 0 & 0 & 0
\end{array}
\right]  ,
\]
we can show by Proposition \ref{exa} below that%
\[
\left\Vert A\right\Vert _{\ast}=\sqrt{26+2\sqrt{20}}<\sqrt{26+2\sqrt{24}}.
\]
In the light of such examples, the main contribution of Theorem \ref{t3} is to
characterize when bound (\ref{mb}) is exact. Although this characterization is
contingent, it does not discard the possibility that the bound is exact for
infinitely many cases. Indeed, proving or disproving the existence of
infinitely many triples of primes $4k-1,$ $6k-1,$ $12k-1$ or $4k+1,$ $6k+1,$
$12k+1$ seems an intractable problem presently. We have checked by a computer
that equality holds in (\ref{mb}) for every $n\leq1000000$ and some
appropriately chosen $m\in\left[  3n+1,4n\right]  .$ For example, for
$n=1000000$ and $m=3597262$, the numbers
\begin{align*}
1199087 &  =4\ast299772-1,\\
\text{ }1798631 &  =6\ast299772-1,\\
3597263 &  =12\ast299772-1
\end{align*}
are primes, and hence equality holds in (\ref{mb}).\medskip

Finally, looking back at Problem \ref{pr1}, we realize that notwithstanding
the precision of Theorems \ref{t1}-\ref{t3}, they shed no light on the order
of magnitude of $\psi_{n}(m)$ when $m$ is large compared to $n$. Hence, we
feel compelled to give the following simple general bound on $\psi_{n}(m)$,
whose proof is given in Section \ref{ur}.

\begin{proposition}
\label{gb}If $n\geq2$ and $1\leq m\leq n^{2},$ then $\psi_{n}(m)\leq\sqrt
{m}+\sqrt{\left\lceil m/n\right\rceil }/2.$
\end{proposition}

The rest of the paper is structured as follows: in Section \ref{np} we give
some definitions, prove two auxiliary results, and prove Proposition \ref{gb}.
In Section \ref{prf}, first we state and prove Theorem \ref{tsm}, a general
structural result about $\left(  0,1\right)  $-matrices with minimum trace
norm, and then we use it to carry out the proofs of Theorems \ref{t1},
\ref{t2}, and \ref{t3}.

\section{\label{np}Notation and preliminaries}

We write $A^{\ast}$ for the Hermitian adjoint of a matrix $A$; in particular,
if $A$ is real, then $A^{\ast}=A^{T}$. The \emph{singular values} of a matrix
$A$ are the square roots of the eigenvalues of $AA^{\ast}$; we denote them as
$\sigma_{1}(A),\sigma_{2}(A),\ldots$ indexed in descending order. \ For more
information on singular values the reader is referred to \cite{HoJo91}. In
particular, we use the following handy interlacing property:\emph{ if }%
$B$\emph{ is a submatrix of }$A$\emph{, then }$\sigma_{i}\left(  A\right)
\geq\sigma_{i}\left(  B\right)  $\emph{ for any admissible }$i$ (see
\cite{HoJo91}, p 149.)

Call two matrices \emph{equivalent} if they can be obtained from each other by
a finite sequence of the following operations:

- transpositions,

- permutations of rows or columns,

- insertions/deletions of zero rows or columns.

Clearly, equivalent matrices have the same nonzero singular values and
therefore the same trace norm.

A $\left(  0,1\right)  $-matrix $A=\left[  a_{i,j}\right]  $ is called a
\emph{step matrix} if $a_{i,j}\geq a_{k,l}$ whenever $i\leq k$ and $j\leq l.$
Thus, a step matrix looks like the following one
\[
\left[
\begin{array}
[c]{cccccc}%
1 & \cdots & 1 & 1 & \ldots & 0\\
1 & \cdots & 1 & 0 & \cdots & 0\\
\cdots & \cdots & \cdots & \cdots & \cdots & \cdots\\
1 & 1 & 0 & \cdots & \cdots & 0\\
\cdots & \cdots & \cdots & \cdots & \cdots & \cdots\\
0 & 0 & \cdots & \cdots & \cdots & 0
\end{array}
\right]  .
\]
The number of different rows (and columns) of a step matrix $A$ is called
\emph{the number of steps} of $A$.

Let us note that step matrices arise in many extremal problems about $\left(
0,1\right)  $-matrices, and they are crucial for our study as well.

Finally, we write $0_{p,q}$ and $J_{p,q}$ for the\ all-zeros and all-ones
matrices of size $p\times q$.

\subsection{\label{ur}Some useful tools}

In this subsection we present two cornerstones for our proofs, and prove
Proposition \ref{gb}.

\begin{proposition}
\label{pro1}If $A$ is a complex matrix with $\sigma_{1}\left(  A\right)  \leq
a\leq\left\vert A\right\vert _{2}$, then
\begin{equation}
\left\Vert A\right\Vert _{\ast}\geq\sqrt{\left\vert A\right\vert _{2}%
^{2}-a^{2}}+a. \label{in}%
\end{equation}

\end{proposition}

\begin{proof}
Let $\sigma_{1},\ldots,\sigma_{n}$ be the \ singular values of $A$. Since
\[
\sigma_{1}^{2}+\cdots+\sigma_{n}^{2}=\mathrm{tr}AA^{\ast}=\left\vert
A\right\vert _{2}^{2},
\]
we see that%
\[
\left(  \left\Vert A\right\Vert _{\ast}-\sigma_{1}\right)  \sigma_{2}%
\geq\sigma_{2}^{2}+\sigma_{3}\sigma_{2}+\cdots+\sigma_{n}\sigma_{2}\geq
\sigma_{2}^{2}+\cdots+\sigma_{n}^{2}=\left\vert A\right\vert _{2}^{2}%
-\sigma_{1}^{2}.
\]
On the other hand, $\sigma_{1}^{2}+\sigma_{2}^{2}\leq\left\vert A\right\vert
_{2}^{2}$, and so
\[
\left(  \left\Vert A\right\Vert _{\ast}-\sigma_{1}\right)  \sqrt{\left\vert
A\right\vert _{2}^{2}-\sigma_{1}^{2}}\geq\left\vert A\right\vert _{2}%
^{2}-\sigma_{1}^{2},,
\]
implying that%
\begin{equation}
\left\Vert A\right\Vert _{\ast}\geq\sigma_{1}+\sqrt{\left\vert A\right\vert
_{2}^{2}-\sigma_{1}{}^{2}}. \label{bin}%
\end{equation}

Next, if $\sigma_{1}\leq\left\vert A\right\vert _{2}/\sqrt{2}$, then%
\[
\sigma_{1}\Vert A\Vert_{\ast}=\sigma_{1}\left(  \sigma_{1}+\cdots+\sigma
_{n}\right)  \geq\sigma_{1}^{2}+\cdots+\sigma_{n}^{2}=\left\vert A\right\vert
_{2}^{2},
\]
and by
\[
\Vert A\Vert_{\ast}^{2}\geq2\left\vert A\right\vert _{2}^{2}=\left(
\sqrt{\left\vert A\right\vert _{2}^{2}-a^{2}}+a\right)  ^{2}+\left(
\sqrt{\left\vert A\right\vert _{2}^{2}-a^{2}}-a\right)  ^{2}\geq\left(
\sqrt{\left\vert A\right\vert _{2}^{2}-a^{2}}+a\right)  ^{2},
\]
inequality (\ref{in}) follows.

On the other hand, the function $f\left(  x\right)  =x+\sqrt{\left\vert
A\right\vert _{2}^{2}-x^{2}}$ is decreasing in $x$ whenever $\left\vert
A\right\vert _{2}/\sqrt{2}<x\leq\left\vert A\right\vert _{2}$. Hence, if
$\sigma_{1}>\left\vert A\right\vert _{2}/\sqrt{2}$, then inequality
(\ref{bin}) implies that
\[
\Vert A\Vert_{\ast}\geq f(\sigma_{1})\geq f\left(  a\right)  =\sqrt{\left\vert
A\right\vert _{2}^{2}-a^{2}}+a
\]
completing the proof of Proposition \ref{pro1}.
\end{proof}

A number of subsequent calculations can be streamlined using the following proposition.\ 

\begin{proposition}
\label{exa} If $A$ is the block matrix%
\[
\left[
\begin{array}
[c]{cc}%
J_{s,p} & J_{s,r}\\
J_{q-s,p} & 0_{q-s,r}%
\end{array}
\right]  ,
\]
then%
\[
\sigma_{2}^{2}\left(  A\right)  =\frac{pq+rs-\sqrt{(pq+rs)^{2}-4(q-s)prs}}{2}%
\]
and
\[
\Vert A\Vert_{\ast}=\sqrt{rs+pq+2\sqrt{prs(q-s)}}.
\]

\end{proposition}

\begin{proof}
It is not hard to see that
\[
AA^{\ast}=\left[
\begin{array}
[c]{cc}%
(p+r)J_{s,s} & pJ_{s,\left(  q-s\right)  }\\
pJ_{\left(  q-s\right)  ,s} & pJ_{\left(  q-s\right)  ,\left(  q-s\right)  }%
\end{array}
\right]  .
\]
Hence, the characteristic polynomial $\phi_{AA^{\ast}}(x)$ of $AA^{\ast}$ is
\[
\phi_{AA^{\ast}}(x)=(-1)^{q}x^{q-2}\left(  x^{2}-(pq+rs)x+prs(q-s)\right)  ,
\]
and consequently the singular values of $A$ satisfy
\begin{align*}
\sigma_{1}^{2}\left(  A\right)   &  =\frac{pq+rs+\sqrt{(pq+rs)^{2}-4(q-s)prs}%
}{2},\\
\sigma_{2}^{2}\left(  A\right)   &  =\frac{pq+rs-\sqrt{(pq+rs)^{2}-4(q-s)prs}%
}{2},\\
\sigma_{3}\left(  A\right)   &  =\cdots=\sigma_{q}\left(  A\right)  =0.
\end{align*}
Therefore,
\begin{align*}
\Vert A\Vert_{\ast}  &  =\sqrt{\sigma_{1}^{2}\left(  A\right)  +\sigma_{2}%
^{2}\left(  A\right)  +2\sigma_{1}\left(  A\right)  \sigma_{2}\left(
A\right)  }\\
&  =\sqrt{pq+rs+2\sqrt{prs(q-s)}},
\end{align*}
as claimed.
\end{proof}

As an immediate application of Proposition \ref{exa}, we prove Proposition
\ref{gb}

\begin{proof}
[\textbf{Proof of Proposition \ref{gb}}]Let $k=\left\lceil m/n\right\rceil $
and $p=\left\lfloor m/k\right\rfloor $; hence $m=kp+s,$ $0\leq s\leq k-1$. If
$s=0,$ then $\psi_{n}(m)=\sqrt{m},$ so suppose that $s>0.$ Obviously the
matrix%
\[
B=\left[
\begin{array}
[c]{cc}%
J_{s,p} & J_{s,1}\\
J_{k-s,p} & 0_{k-s,1}%
\end{array}
\right]
\]
can be completed by zero rows and columns to a matrix in $\mathbb{Z}%
_{n}\left(  m\right)  $; hence, Proposition \ref{exa} implies that
\begin{align*}
\psi_{n}(m)  &  \leq\left\Vert B\right\Vert _{\ast}=\sqrt{m+2\sqrt{ps\left(
k-s\right)  }}\leq\sqrt{m+\sqrt{pk^{2}}}\\
&  <\sqrt{m+\sqrt{m\left\lceil m/n\right\rceil }}<\sqrt{m}+\sqrt{\left\lceil
m/n\right\rceil }/2,
\end{align*}
as claimed.
\end{proof}

\section{\label{prf}Proofs}

The proofs of Theorems \ref{t1}, \ref{t2} and \ref{t3} are essentially deduced
from a general statement given in Theorem \ref{tsm} below. Its proof is
somewhat involved, and in fact many calculations could have been spared had we
readily used MATLAB or a similar package. However, except on one occasion, we
need exact values, so to avoid any doubt of rounding errors we give explicit
calculations that can be verified directly.

\begin{theorem}
\label{tsm}Let $A$ be a $\left(  0,1\right)  $-matrix with $m$ ones. If
\[
\sqrt{m}<\left\Vert A\right\Vert _{\ast}<\sqrt{m-1}+1,
\]
then $A$ is equivalent to a block matrix of the type
\[
\left[
\begin{array}
[c]{cc}%
J_{s,p} & J_{s,1}\\
J_{r,p} & 0_{r,1}%
\end{array}
\right]  .
\]

\end{theorem}

\begin{proof}
The main idea of the proof is that $A$ cannot contain a submatrix $X$ with
$\sigma_{2}\left(  X\right)  \geq1$ or $\sigma_{2}^{2}\left(  X\right)
+\sigma_{3}^{2}\left(  X\right)  \geq1,$ because then $\sigma_{2}\left(
A\right)  \geq1$ or $\sigma_{2}^{2}\left(  A\right)  +\sigma_{3}^{2}\left(
A\right)  \geq1$, and so $\sigma_{1}^{2}\left(  A\right)  \leq m-1.$ Hence,
Proposition \ref{pro1} implies that $\left\Vert A\right\Vert _{\ast}\geq
\sqrt{m-1}+1.$

First, note that $A$ does not contain the matrices
\[
X_{1}=\left[
\begin{array}
[c]{cc}%
1 & 0\\
0 & 1
\end{array}
\right]  \text{ \ \ \ or \ \ \ }X_{2}=\text{\ }\left[
\begin{array}
[c]{cc}%
0 & 1\\
1 & 0
\end{array}
\right]  ,
\]
because $\sigma_{2}\left(  X_{1}\right)  =\sigma_{2}\left(  X_{2}\right)  =1$.

Next, delete the zero rows and columns of $A,$ and let $t\times q$ be the size
of the resulting matrix $A^{\prime}=[a_{i,j}^{\prime}]$. Given two vectors
$\mathbf{x}=\left(  x_{1},\ldots,x_{n}\right)  $ and \ $\mathbf{y}=\left(
y_{1},\ldots,y_{n}\right)  ,$ write $\mathbf{x}\succ\mathbf{y}$ if $x_{i}\geq
y_{i}$ for all $i\in\left[  n\right]  .$

Let $\mathbf{r}_{1},\ldots,\mathbf{r}_{t}$ be the row vectors of $A^{\prime}$.
It is not hard to see that if $1\leq i<j\leq t,$ then either $\mathbf{r}%
_{i}\succ\mathbf{r}_{j}$ or $\mathbf{r}_{j}\succ\mathbf{r}_{i}$, for otherwise
$A^{\prime}$ contains either $X_{1}$\ or $X_{2}$. Hence, we can permute the
rows of $A^{\prime}$ so that $\mathbf{r}_{1}\succ\cdots\succ\mathbf{r}_{t}$.
Applying the same argument to the columns of $A^{\prime},$ we can ensure that
$a_{i,j}^{\prime}\geq a_{k,l}^{^{\prime}}$ whenever $i\leq k$ and $j\leq l$.
Therefore $A$ is equivalent to the step matrix $A^{\prime}$.

Since for $t\leq2$ or $q\leq2$ the statement is obvious, we assume that
$t\geq3$ and $q\geq3$. Suppose first that $q\geq4$ and that $A^{\prime}$ has
at least three steps. Then $A$ must contain a matrix of the type%
\[
X=\left[
\begin{array}
[c]{cccc}%
1 & 1 & 1 & 1\\
1 & a & b & 0\\
1 & c & 0 & 0
\end{array}
\right]  ,
\]
where $a,b,c$ can be zero or one. If $b=0,$ then $A^{\prime}$ contains the
matrix%
\[
X_{3}=\left[
\begin{array}
[c]{ccc}%
1 & 1 & 1\\
1 & 0 & 0\\
1 & 0 & 0
\end{array}
\right]  ,
\]
for which Proposition \ref{exa} gives $\sigma_{2}\left(  X_{3}\right)  =1$, a
contradiction. Therefore, $b=1$ and $A^{\prime}$ contains one of the matrices
\[
X_{4}=\left[
\begin{array}
[c]{cccc}%
1 & 1 & 1 & 1\\
1 & 1 & 1 & 0\\
1 & 1 & 0 & 0
\end{array}
\right]  \text{ \ \ or \ \ }X_{5}=\left[
\begin{array}
[c]{cccc}%
1 & 1 & 1 & 1\\
1 & 1 & 1 & 0\\
1 & 0 & 0 & 0
\end{array}
\right]  .
\]
We shall prove that $\sigma_{2}^{2}\left(  X_{4}\right)  +\sigma_{3}%
^{2}\left(  X_{4}\right)  >1$ and $\sigma_{2}^{2}\left(  X_{5}\right)
+\sigma_{3}^{2}\left(  X_{5}\right)  >1.$ Indeed,
\[
X_{4}X_{4}^{\ast}=\left[
\begin{array}
[c]{ccc}%
4 & 3 & 2\\
3 & 3 & 2\\
2 & 2 & 2
\end{array}
\right]  ,
\]
and the characteristic polynomial of $X_{4}X_{4}^{\ast}$ is
\[
\phi_{X_{4}X_{4}^{\ast}}\left(  x\right)  =x^{3}-9x^{2}+9x-2=\left(
x-8\right)  \left(  x^{2}-x+1\right)  +6.
\]
Hence $\phi_{X_{4}X_{4}^{\ast}}\left(  x\right)  >0$ if $x\geq8,$ and so
$\sigma_{1}^{2}\left(  X_{4}\right)  <8.$ Therefore,
\[
\sigma_{2}^{2}\left(  X_{4}\right)  +\sigma_{3}^{2}\left(  X_{4}\right)
=\mathrm{tr}X_{4}X_{4}^{\ast}-\sigma_{1}^{2}\left(  X_{4}\right)  >9-8=1,
\]
a contradiction.

Likewise, we see that
\[
X_{5}X_{5}^{\ast}=\left[
\begin{array}
[c]{ccc}%
4 & 3 & 1\\
3 & 3 & 1\\
1 & 1 & 1
\end{array}
\right]
\]
and the characteristic polynomial of $X_{5}X_{5}^{\ast}$ is
\[
\phi_{X_{5}X_{5}^{\ast}}\left(  x\right)  =x^{3}-8x^{2}+8x-2=\left(
x-7\right)  \left(  x^{2}-x+1\right)  +5.
\]
Hence $\phi_{X_{5}X_{5}^{\ast}}\left(  x\right)  >0$ if $x\geq7$, and so
$\sigma_{1}^{2}\left(  X_{5}\right)  <7$. Therefore,%
\[
\sigma_{2}^{2}\left(  X_{5}\right)  +\sigma_{3}^{2}\left(  X_{5}\right)
=\mathrm{tr}X_{5}X_{5}^{\ast}-\sigma_{1}^{2}\left(  X_{5}\right)  >8-7=1,
\]
a contradiction.

Thus, we have proved that $A^{\prime}$ is always a step matrix with at most
two steps, except if $q=3.$ The same argument applies if $t\geq4$. so it
remains the case $t=3$ and $q=3.$ But in that case the only matrix with three
steps is
\[
B=\left[
\begin{array}
[c]{ccc}%
1 & 1 & 1\\
1 & 1 & 0\\
1 & 0 & 0
\end{array}
\right]  ,
\]
and using MATLAB, one can see that $\left\Vert B\right\Vert _{\ast}>\sqrt
{5}+1$, contrary to the premises of the theorem.

Therefore, $A^{\prime}$ has at most two steps, that is to say, it has the
block form
\[
A^{\prime}=\left[
\begin{array}
[c]{cc}%
J_{s,p} & J_{s,k}\\
J_{r,p} & 0_{r,k}%
\end{array}
\right]  .
\]
Clearly, $k\geq1$ and $r\geq1$, as otherwise $\left\Vert A^{\prime}\right\Vert
_{\ast}=\sqrt{m}$, contrary to the premise $\left\Vert A^{\prime}\right\Vert
_{\ast}>\sqrt{m}$.

Finally, if $r\geq2$ and $k\geq2$, then $A^{\prime}$ contains the matrix
$X_{3},$ which is a contradiction since $\sigma_{2}\left(  X_{3}\right)  =1$.
Therefore, either $k=1$ or $r=1,$ completing the proof of Theorem \ref{tsm}.
\end{proof}

\subsection{Proof of Theorem \ref{t1}}

\begin{proof}
Let $n\geq2$ and $n<m\leq2n.$ If $m$ is not a prime number, then $m=ab$ for
some integers $a\geq2$ and $b\geq2.$ Hence $a\leq n$ and $b$ $\leq n$;
therefore, $\mathbb{Z}_{n}\left(  m\right)  $ contains a matrix of rank $1$
and $\psi_{n}\left(  m\right)  =\sqrt{m}.$

Suppose now that $m$ is a prime, which must be odd, because $m>n\geq2.$ \ Let
$m=2k+1$ and apply Proposition \ref{exa} to the matrix
\[
B=\left[
\begin{array}
[c]{cc}%
J_{1,k} & 1\\
J_{1,k} & 0
\end{array}
\right]  .
\]
We get $\left\Vert B\right\Vert _{\ast}=\sqrt{m+\sqrt{2\left(  m-1\right)  }%
}.$ Since $B$ can be extended to an $n\times n$ matrix by addition of zero
rows and columns, we see that
\begin{equation}
\psi_{n}\left(  m\right)  \leq\sqrt{m+\sqrt{2\left(  m-1\right)  }}.
\label{ub}%
\end{equation}

Let $A\in\mathbb{Z}_{n}\left(  m\right)  $ be a matrix with $\left\Vert
A\right\Vert _{\ast}=\psi_{n}\left(  m\right)  $. We complete the proof by
showing that $A$ is equivalent to $B.$ Since
\[
\sqrt{m+\sqrt{2\left(  m-1\right)  }}<\sqrt{m-1}+1,
\]
Theorem \ref{tsm} implies that $A$ is equivalent to a matrix
\[
A^{\prime}=\left[
\begin{array}
[c]{cc}%
J_{s,t} & J_{s,1}\\
J_{r,t} & 0_{r,1}%
\end{array}
\right]  .
\]
To finish the proof we show that $r=1,$ and either $s=1$ or $t=1$. Indeed if
$r\geq2,$ then $A$ contains the matrix
\[
Y_{1}=\left[
\begin{array}
[c]{cc}%
1 & 1\\
1 & 0\\
1 & 0
\end{array}
\right]  .
\]
Proposition \ref{exa} implies that
\[
\sigma_{2}^{2}\left(  Y_{1}\right)  =\frac{2+2-\sqrt{(2+2)^{2}-4(1)2}}%
{2}=\sqrt{2-\sqrt{2}}>\frac{1}{2},
\]
and Proposition \ref{pro1} implies that
\[
\left\Vert A\right\Vert _{\ast}\geq\sqrt{m-1/2}+1/\sqrt{2}>\sqrt
{m+\sqrt{2\left(  m-1\right)  }},
\]
contrary to (\ref{ub}). Hence $r=1.$

Finally, if both $s\geq2$ and $t\geq2,$ then $A^{\prime}$ contains the matrix
\[
Y_{2}=\left[
\begin{array}
[c]{ccc}%
1 & 1 & 1\\
1 & 1 & 1\\
1 & 1 & 0
\end{array}
\right]  .
\]
Applying Proposition \ref{exa}, we find that%
\[
\sigma_{2}^{2}\left(  Y_{2}\right)  =\frac{8-\sqrt{64-16}}{2}=4-2\sqrt
{3}>\frac{1}{2}.
\]
Hence, Proposition \ref{pro1} implies that
\[
\left\Vert A\right\Vert _{\ast}=\left\Vert A^{\prime}\right\Vert _{\ast}%
\geq\sqrt{m-1/2}+1/\sqrt{2}>\sqrt{m+\sqrt{2\left(  m-1\right)  }},
\]
contrary to (\ref{ub}). The proof of Theorem \ref{t1} is completed.
\end{proof}

\subsection{Proof of Theorem \ref{t2}}

\begin{proof}
Let $n\geq4$ and $2n<m\leq3n.$ It $m$ has three prime factors $p\leq q\leq r$,
then choosing $a=pq$ and $b=m/a$, we see that $3\leq a\leq n$ and $3\leq b\leq
n$, because $m>8$; hence $\mathbb{Z}_{n}\left(  m\right)  $ contains a matrix
of rank $1$ and $\psi_{n}\left(  m\right)  =\sqrt{m}$.

It $m$ has only two prime factors $p\leq q,$ then choosing $a=p$ and $b=m/a,$
we see that $\mathbb{Z}_{n}\left(  m\right)  $ contains a matrix of rank $1$
and $\psi_{n}\left(  m\right)  =\sqrt{m},$ unless $p=2$, that is to say,
unless $m$ is a double of a prime.

Now, let $m$ be a prime or a double of prime, and suppose that $m=3k+s,$ where
$1\leq s\leq2.$ Applying Proposition \ref{exa} to the matrix
\[
B=\left[
\begin{array}
[c]{cc}%
J_{s,k} & J_{s,1}\\
J_{3-s,k} & 0_{3-s,1}%
\end{array}
\right]  ,
\]
we get
\[
\left\Vert B\right\Vert _{\ast}=\sqrt{m+2\sqrt{2\left(  m-s\right)  /3}}%
=\sqrt{m+2\sqrt{2\left\lfloor m/3\right\rfloor }}.
\]
Since $B$ can be extended to an $n\times n$ matrix by addition of zero rows
and columns, we see that
\begin{equation}
\psi_{n}\left(  m\right)  \leq\sqrt{m+2\sqrt{2\left\lfloor m/3\right\rfloor }%
}. \label{ub2}%
\end{equation}

Let $A\in\mathbb{Z}_{n}\left(  m\right)  $ be a matrix with $\left\Vert
A\right\Vert _{\ast}=\psi_{n}\left(  m\right)  $. We shall show that $A$ is
equivalent to $B.$ Since
\[
\sqrt{m+2\sqrt{2\left\lfloor m/3\right\rfloor }}<\sqrt{m-1}+1,
\]
Theorem \ref{tsm} implies that $A$ is equivalent to a step matrix of the type%
\[
A^{\prime}=\left[
\begin{array}
[c]{cc}%
J_{s,p} & J_{s,1}\\
J_{r,p} & 0_{r,1}%
\end{array}
\right]  .
\]
If $s+r=3,$ we see that $A^{\prime}=B$, proving the theorem.

Assume now that $s+r\geq4$; we shall show that this assumption contradicts
(\ref{ub2}). First, note that $m>2n$ implies that $p\geq2$. If $r\geq2$,
$A^{\prime}$ contains the matrix
\[
Y_{2}=\left[
\begin{array}
[c]{ccc}%
1 & 1 & 1\\
1 & 1 & 1\\
1 & 1 & 0\\
1 & 1 & 0
\end{array}
\right]  ,
\]
and Proposition \ref{exa} implies that
\[
\sigma_{2}^{2}\left(  A\right)  \geq\sigma_{2}^{2}\left(  Y_{2}\right)
=\frac{10-\sqrt{68}}{2}>\frac{2}{3}.
\]
We see that $\sigma_{1}^{2}\leq m-2/3$ and Proposition \ref{pro1} gives
\[
\left\Vert A\right\Vert _{\ast}>\sqrt{m-2/3}+\sqrt{2/3}>\sqrt{m+2\sqrt
{2\left\lfloor m/3\right\rfloor }},
\]
contradicting (\ref{ub2}). Hence, $r=1$, and so
\[
A^{\prime}=\left[
\begin{array}
[c]{cc}%
J_{s,p} & J_{s,1}\\
J_{1,p} & 0
\end{array}
\right]  .
\]
Clearly, we may suppose that $p\geq s,$ because $p$ and $s$ are symmetric
parameters in the shape of $A^{\prime}$.

Next, Proposition \ref{exa} gives
\[
\left\Vert A^{\prime}\right\Vert _{\ast}=\sqrt{m+2\sqrt{\frac{\left(
m-s+1\right)  s}{s+1}}}.
\]
We see that
\begin{equation}
\frac{\left(  m-s+1\right)  s}{s+1}=m+2-s-\frac{m+2}{s+1}, \label{exp}%
\end{equation}
and the derivative of this expression as a function of $s$ is
\[
-1+\frac{m+2}{\left(  s+1\right)  ^{2}}.
\]
Therefore, expression (\ref{exp}) is increasing with $s$, because%
\[
m+2=p\left(  s+1\right)  +s+2\geq s\left(  s+1\right)  +s+2>\left(
s+1\right)  ^{2}.
\]
Hence,%
\[
\frac{\left(  m-s+1\right)  s}{s+1}\geq\frac{\left(  m-2\right)  3}{4}%
>\frac{\left(  m-1\right)  2}{3},
\]
and so,
\[
\left\Vert A^{\prime}\right\Vert _{\ast}>\sqrt{m+2\sqrt{2\left[  m/3\right]
}},
\]
contradicting (\ref{ub2}). Theorem\ \ref{t2} is proved.
\end{proof}

\subsection{Proof of Theorem \ref{t3}}

The starting point of the proof of Theorem \ref{t3} are the following estimates:

\begin{proposition}
\label{pro4}If $m\leq4n,$ then
\[
\psi_{n}\left(  m\right)  \leq\left\{
\begin{array}
[c]{ll}%
\sqrt{m+2\sqrt{3\left(  m-1\right)  /4}}\text{,} & \text{if }m\equiv1\text{
}\left(  \operatorname{mod}4\right)  \text{;}\\
\sqrt{m+2\sqrt{3\left(  m-3\right)  /4}}\text{,} & \text{if }m\equiv3\text{
}\left(  \operatorname{mod}4\right)  \text{;}\\
\sqrt{m+2\sqrt{m-2}}\text{,} & \text{if }m\equiv2\text{ }\left(
\operatorname{mod}4\right)  \text{.}%
\end{array}
\right.
\]

\end{proposition}

\begin{proof}
Suppose that $m=4k+s,$ where $k$ is an integer and $1\leq s\leq3.$ Then
$k+1\leq n,$ and so the $\left(  0,1\right)  $-matrix
\[
B=\left[
\begin{array}
[c]{cc}%
J_{s,k} & J_{s,1}\\
J_{4-s,k} & 0_{4-s,1}%
\end{array}
\right]
\]
can be completed with zero rows and columns to a matrix $A\in\mathbb{Z}%
_{n}\left(  m\right)  $. Applying Proposition \ref{exa}, we find that
\[
\left\Vert A\right\Vert _{\ast}=\left\Vert B\right\Vert _{\ast}\leq\left\{
\begin{array}
[c]{ll}%
\sqrt{m+2\sqrt{3\left(  m-1\right)  /4}}\text{,} & \text{if }m\equiv1\text{
}\left(  \operatorname{mod}4\right)  \text{;}\\
\sqrt{m+2\sqrt{3\left(  m-3\right)  /4}}\text{,} & \text{if }m\equiv3\text{
}\left(  \operatorname{mod}4\right)  \text{;}\\
\sqrt{m+2\sqrt{m-2}}\text{,} & \text{if }m\equiv2\text{ }\left(
\operatorname{mod}4\right)  \text{.}%
\end{array}
\right.  ,
\]
completing the proof of Proposition \ref{pro4}.
\end{proof}

In the proof of Theorem \ref{t3} we shall use the following universal bound,
which follows from Proposition \ref{pro4} by an easy calculation.

\begin{corollary}
\label{cor1}If $m\leq4n,$ then
\[
\psi_{n}\left(  m\right)  <\sqrt{m-1}+1\text{.}%
\]

\end{corollary}

\medskip

Having Proposition \ref{pro4} and Theorem \ref{tsm} in hand, we begin the
proof of Theorem \ref{t3}.

\begin{proof}
[\textbf{Proof of Theorem \ref{t3}}]Proposition \ref{pro4} and an easy
calculation show that inequality (\ref{mb}) always holds. Thus, we move
forward to the condition for equality in (\ref{mb}). Suppose that $n\geq5,$
$3n<m\leq4n,$ and
\begin{equation}
\psi_{n}\left(  m\right)  =\sqrt{m+2\sqrt{m-2}}. \label{equ}%
\end{equation}
In view of Proposition \ref{pro4}, we see that $m=4k+2.$

Our first goal is to show that there exists a positive integer $k$ such that either:

(a) $m=12k+2$ and $4k+1,$ $6k+1,$ $12k+1$ are primes, or

(b) $m=12k-2$ and $4k-1,$ $6k-1,$ $12k-1$ are primes.

To this end, we shall ptove the following claim:\medskip

\textbf{Claim A }\emph{The equation}%
\begin{equation}
m=ab+c \label{deq}%
\end{equation}
\emph{has no solution in integers }$a,b,c$\emph{ such that }$a\geq b\geq
5$\emph{ and }$c\in\left\{  -1,0,1\right\}  $\emph{. }\medskip

\emph{Proof. }Indeed, assume that $a,b,c$ is such a solution. \ Note first
that
\begin{equation}
b\leq a=\left(  m-c\right)  /b\leq\left(  4n+1\right)  /5<n. \label{dim}%
\end{equation}
For convenience, we consider the cases $c=0,1,-1$ separately.

If $c=0$, inequality (\ref{dim}) implies that there exists a matrix
$A\in\mathbb{Z}_{n}\left(  m\right)  $ that is equivalent to $J_{a,b}$, and
consequently $\psi_{n}\left(  m\right)  =\sqrt{m},$ contrary to (\ref{equ}).

If $c=1$, inequality (\ref{dim}) implies that there exists a matrix
$A\in\mathbb{Z}_{n}\left(  m\right)  $ that is equivalent to%
\[
\left[
\begin{array}
[c]{cc}%
J_{1,a} & 1\\
J_{b-1,a} & 0_{b-1,1}%
\end{array}
\right]  ,
\]
and Proposition \ref{exa} implies that
\[
\left\Vert A\right\Vert _{\ast}=\sqrt{m+2\sqrt{\frac{\left(  m-1\right)
\left(  b-1\right)  }{b}}}\leq\sqrt{m+2\sqrt{\frac{\left(  m-1\right)  4}{5}}%
}<\sqrt{m+2\sqrt{m-2}},
\]
contrary to (\ref{equ}).

Finally, if $c=-1$, inequality (\ref{dim}) implies that there exists a matrix
$A\in\mathbb{Z}_{n}\left(  m\right)  $ that is equivalent to%
\[
\left[
\begin{array}
[c]{cc}%
J_{b-1,a-1} & J_{b-1,1}\\
J_{1,a-1} & 0
\end{array}
\right]  ,
\]
and Proposition \ref{exa} implies that
\[
\left\Vert A\right\Vert _{\ast}=\sqrt{m+2\sqrt{\frac{\left(  m-b+1\right)
\left(  b-1\right)  }{b}}}\leq\sqrt{m+2\sqrt{\frac{\left(  m-1\right)  4}{5}}%
}<\sqrt{m+2\sqrt{m-2}},
\]
contrary to (\ref{equ}). Claim A is proved.$\hfill\square\medskip$

To reveal the consequences of Claim A, consider $m$ modulo $12.$ Clearly,
either $m=12k+2$ or $m=12k-2$ for some positive integer $k$, for if $m=12k+6$
for some integer $k\geq2$, then (\ref{deq}) has a solution with $a\geq b\geq5$
and $c=0$, contradicting Claim A.

Let us consider the case $m=12k+2$ in detail. First, note that $6k+1$ must be
a prime, as if $6k+1=ab$ for some integers $a\geq$ $b\geq2$, then $b\geq5,$
and letting $x=2a$ and $y=b,$ we see that $m=xy$ and $x\geq y\geq5$, contrary
to Claim A.

Next, if $4k+1$ is not a prime, say $4k+1=ab$ for some integers $a\geq$
$b\geq2$, then $b\geq5,$ and letting $x=3a$ and $y=b,$ we see that $m=xy-1,$
$x\geq y\geq5,$ contrary to Claim A.

Finally, if $12k+1$ is not a prime, say $12k+1=ab$ for some integers $a\geq
b\geq2$, then $b\geq5$, and letting $x=a$ and $y=b,$ we see that $m=xy+1,$
$x\geq y\geq5,$ contrary to Claim A.

Thus if $m=12k+2,$ then $4k+1,$ $6k+1$ and $12k+1$ are primes.

By the same argument we find that if $m=12k-2,$ then $4k-1,$ $6k-1$ and
$12k-11$ are primes.

It remains to prove the converses of the above implications. In particular,
let $m=12k+2,$ and let $4k+1,$ $6k+1$ and $12k+1$ be primes. We have to prove
that (\ref{equ}) holds. To this end, let $A\in\mathbb{Z}_{n}\left(  m\right)
$ be such that%
\[
\left\Vert A\right\Vert _{\ast}=\psi_{n}\left(  m\right)  \leq\sqrt
{m+2\sqrt{m-2}}.
\]
Clearly, $\left\Vert A\right\Vert _{\ast}>\sqrt{m},$ for otherwise there exist
integers $a$ and $b$ such that $m=ab$, $4\leq a\leq n$, and $4\leq b\leq n,$
which is a contradiction, as $m=2\left(  6k+1\right)  $ and $6k+1$ is a prime.
Now Corollary \ref{cor1} and Theorem \ref{tsm} imply that $A$ is equivalent to
a step matrix
\[
A^{\prime}=\left[
\begin{array}
[c]{cc}%
J_{s,p} & J_{s,1}\\
J_{r,p} & 0_{r,1}%
\end{array}
\right]  .
\]
Further, the premise $m>3n$ implies that $s+r\geq4.$ Our last goal is to show
that $r+s=4.$ Assume for a contradiction that $r+s\geq5$. Since $p\geq
\left\lfloor m/n\right\rfloor \geq3,$ if $s\neq1$ and $r\neq1,$ then
$A^{\prime}$ contains one of the matrices
\[
Z_{1}=\left[
\begin{array}
[c]{cccc}%
1 & 1 & 1 & 1\\
1 & 1 & 1 & 1\\
1 & 1 & 1 & 0\\
1 & 1 & 1 & 0\\
1 & 1 & 1 & 0
\end{array}
\right]  \text{ \ or \ }Z_{2}=\left[
\begin{array}
[c]{cccc}%
1 & 1 & 1 & 1\\
1 & 1 & 1 & 1\\
1 & 1 & 1 & 1\\
1 & 1 & 1 & 0\\
1 & 1 & 1 & 0
\end{array}
\right]  .
\]
Using Proposition \ref{exa}, we find that
\begin{align*}
\sigma_{2}^{2}\left(  Z_{1}\right)   &  =\frac{17-\sqrt{17^{2}-4\cdot
3\cdot2\cdot3}}{2}=\frac{17-\sqrt{217}}{2}>\frac{17-15}{2}=1\text{;}\\
\sigma_{2}^{2}\left(  Z_{2}\right)   &  =\frac{18-\sqrt{18^{2}-4\cdot
2\cdot3\cdot3}}{2}=\frac{18-\sqrt{252}}{2}>\frac{18-16}{2}=1.
\end{align*}
Hence, $\sigma_{1}\left(  A\right)  \leq m-1$ and Proposition \ref{pro1}
implies that $\left\Vert A\right\Vert _{\ast}>\sqrt{m-1}+1$, contradicting
Corollary \ref{cor1}. Therefore either $s=1$ or $r=1$. If $s=1$, then
$m=p\left(  r+s\right)  +1,$ which is a$.$contradiction, as $12k+1$ is a
prime. If $r=1,$ then $m+1=\left(  r+s\right)  \left(  p+1\right)  ,$ which is
a contradiction, as $r+s\geq5$ and $p+1\geq4,$ whereas $3$ and $4k+1$ are the
only divisors of $m+1$.

Thus, we see that $s+r=4.$ Since $s\equiv m$ $\left(  \operatorname{mod}%
4\right)  $, we get $s=2$, and Proposition \ref{exa} implies that (\ref{equ}) holds.

The same argument show that if $m=12k-2,$ and $4k-1,$ $6k-1,$ $12k-1$ are
primes, then (\ref{equ}) holds. Theorem \ref{t3} is proved.
\end{proof}

\bigskip

\textbf{Acknowledgement }This work has been done while the second author was
visiting the Department of Mathematical Sciences at the University of Memphis
in 2016/2017. Her stay was supported by Colciencias 727.

\end{document}